\documentclass[12pt]{amsart}
\usepackage{graphicx}
\usepackage{amsmath}
\usepackage{amssymb}
\usepackage{tikz-cd}

\newtheorem{theorem}{Theorem}[section]

\newtheorem{corollary}[theorem]{Corollary}

\theoremstyle{definition}
\newtheorem{definition}[theorem]{Definition}

\newtheorem{example}[theorem]{Example}

\numberwithin{equation}{section}
\numberwithin{equation}{section}

\begin{document}

\title {essential ideal transforms}

\author[R. H. Mustafa]{Runak H. Mustafa}
\address{Mathematics Department, Faculty of Science\\
	Soran University\\ 44008, Soran, Erbil
	Kurdistan Region, Iraq}
\email {rhm310h@maths.soran.edu.iq}

\author[I. Akray]{Ismael Akray}
\address{Mathematics Department, Faculty of Science\\
	Soran University\\ 44008, Soran, Erbil
	Kurdistan Region, Iraq}

\email {ismael.akray@soran.edu.iq}

\begin{abstract}
It is our intention in this research generalized some concept in local cohomology such as contravarint functor $ext$, covariant functor $Ext$, covarian functor $Tor$ and ideal transforms with $e$-exact sequences. The $e$-exact sequence was introduced by Akray and Zebari \cite{AZ} in 2020. We obtain for a torsion-free modules $B$, $_eext^n_R(P,B)=0$ while $_eExt^n_R(A,E)=0$ for every module $A$. Also for any torsion-free module $B$ we have an $e$-exact sequence $0\to \Gamma_{a}(B) \to B\to D_{a}(B)\to H^1_{a}(B)\to 0$ and an isomorphisms between $B$ and $r D_{a}(B)$. Finally  we generalize Mayer-Vietories with $e$-exact sequences in essential local cohomology, we get a special $e$-exact sequences.
\end{abstract}

\maketitle

\section{Introduction}
  Throughout this article, $R$ will denote a Noetherian domain and $B$ torsion free e-enjective $R$-module.  In 1972, R. S. Mishra introduced a generalization for split sequence where a semi-sequence $M_{i-1}\stackrel{f_{i-1}}\to M_i\stackrel{f_i}\to M_{i+1}$ is called semi-split if $Ker(f_i)$ is a direct summand of $M_i$ \cite{H}. So a semi-split is split if and only if it is an exact. In 1999, Davvaz and parnian-Goramaleky introduced a generalization for exact sequences called it a $U$-exact sequence \cite {DP}. A submodule $N$ of an $R$-module $M$ is called essential or large in $M$ if it has non-zero intersection with every non-zero submodule of $M$ and denoted by $N\leqslant_e M$. Akray and Zebari in 2020 \cite{AZ} gives another generalization to exact sequences of modules and instead of the equality of $Im(f)$ with $Ker(g)$ they took $Im(f)$ as a large (essential) submodule of $Ker(g)$ in a sequence $A\stackrel{f} \to D\stackrel{g}\to C$ and called it essential exact sequence  or simply $e$-exact sequence. Equivalently, a sequence of $R$-modules and $R$-morphisms $\dots\to N_{i-1} \stackrel{f_{i-1}}\to N_i\stackrel{f_i}\to N_{i+1}\to \dots$ is said to be essential exact (e-exact) at $N_i$, if $Im(f_{i-1})\leqslant_e Ker(f_i)$ and to be e-exact if it is an e-exact at $N_i$ for all $i$. In particular, a sequence of $R$-modules and $R$-morphisms $0\to L\stackrel{f_1} \to M\stackrel{f_2}\to N\to 0$ is a short e-exact sequence if and only if  $Ker(f_1)=0, Im(f_1)\leqslant_e Ker(f_2)$ and $Im(f_2)\leqslant_e N$. They studied some basic properties of $e$-exact sequences and established their connection with notions in module theory and homological algebra \cite{ZA}. Also, F. Campanini and A. Facchini were worked on $e$-exact sequences and studied the relation of $e$-exactness with some related functors like the functor defined on the category of $R$-modules to the spectral category of $R$-modules and the localization functor with respect to the singular torsion theory \cite{CF}. Furthermore, Akray and R. H. Mustafa in 2023 introduced and proved further properties of $e$-exact sequences and we will restrict our discussion to their applications on both injective modules and the torsion functor of local cohomology \cite{MA}.The local cohomology was introduced by Grothendieck in a seminar in Harvard 1961 and written up by Hartshorne in 1967. Next, this subject was studied by Hartshorne and numerous authors even in the recent years see \cite{R}, \cite{MR} and \cite{G}.

  In this research we generalize some concept in local cohomology such as contravarint functor $_eext$, covariant functor $_e{Ext}$ and ideal transforms with $e$-exact sequences.
  
   In section two, we describe the concept $_eext_n^R$ and $_eExt_n^R$ as well as we investigate the propertes each of them. For example, let $0\to A'\to A\to A''\to 0$ be an $e$-exact sequence of $R$-modules , then there is a long $e$-exact sequence $0\to r_1Hom(A'',B)\to r_2Hom(A,B)\to r_3Hom(A',B)\to {_eext^1_R(A''.B)}\to {_eext^1_R(A,B)}\to \dots $ for some nonzero element $r_1, r_2, r_3\in R$ and also we have for any $R$-module $A$ and $E$ is an $e$--injective $R$-module, $_eExt^n_R(A,E)=0$, for all $n\geq1$.

  In section three, we construct essential ideal transform and we find the new $e$-exact sequences by generalized the idea of Mayer-vietores sequence. Also we prove that for any torsion-free $R$-module $B$ there exists $0\ne r\in R$ such that $\epsilon^*: B\to rD_{a}(B)$ is an isomorphism if and only if $\Gamma_{a}(B)= H^1_{a}(B)=0$ and also we show that there is an $e$-exact sequences $0\to D_{r(a+b)}(B)\to D_a(B)\bigoplus D_b(B)\to D_{a\cap b}(B)\to rH^2_{r(a+b)}(B)\to rH^2_{a}(B)\bigoplus rH^2_{b}(B)\to rH^2_{a\cap b}(B)\to \dots$.  
  
  \hskip 5cm
\section{contravariant and covariant right essential derived functors}
\subsection{Contravariant essential derived functor}
In this subsection we want to describe contravariant right derived functors $ext^n_R$ on the $e$-projective resolutions call it essential derived functors (berifly $_eext_n^R)$ and discuss some properties of them. On the other hand, we present some definition that are central for our object such as essential injective and essential projective module as the following:
\begin{definition}\cite{MA}
	An $R$-module $E$ is an e-injective if satisfies the following condition: for any monic $f_1:A_1\rightarrow A_2$
	and any map $f_2:A_1\rightarrow E$, there exist $0\ne r \in R$ and $f_3:A_2\rightarrow E$ such that $f_3f_1=rf_2$.
	\[\begin{tikzcd}[sep=2cm]
		{} & E & {} \\
		0 \ar{r} & A_1 \ar{u}{f_2} \ar[r,swap]{r}{f_1} & A_2 \ar[dashed,swap]{lu}{f_3}
	\end{tikzcd}\]
	In this case, we say the map $f_3$ is essentially extends to the map $f_2$. An $e$-injective module may not be injective, for example the $\mathbb{Z}$-module $\mathbb{Z}$ is $e$-injective module, but it is not injective see \cite[Example 2.3]{MA}. 	
\end{definition}

\begin{definition} {\cite{AZ}}
	An $e$-exact sequence $0\rightarrow A\stackrel i\to B\stackrel p\to C\to 0$ is $e$-split if there exist $0\ne s \in R$ and a morphism $ j:C \rightarrow B $\:(or $f:B \rightarrow A $) such that $ pj=sI_C$ (or $\: fi=sI_A)$.
\end{definition}
\begin{definition}\cite{AZ}
	An $R$-module $P$ is $e$-projective if it satiesfies the following condition: for any $e$-epic map  $f_1:A_1\rightarrow A_2$
	and any map $f_2:P\rightarrow A_2$, there exist $0\ne r \in R$ and $f_3:P\rightarrow A_1$ such that $f_1f_3=rf_2$.
	\[\begin{tikzcd}[sep=2cm]
		{} & P\ar{d}{f_2} \ar[dashed,swap]{ld}{f_3} & {} \\
		A_1 \ar{r}{f_1} & A_2 \ar{r} & 0
	\end{tikzcd}\]
\end{definition}
The following example shows that an $e$-projective module may not be projective.
\begin{example}
	Consider we have an $e$-exact sequence of $\frac{\mathbb{Z}}{16\mathbb{Z}}$-modules $0\to \frac{4\mathbb{Z}}{16\mathbb{Z}}\stackrel{f}\to \frac{\mathbb{Z}}{16\mathbb{Z}}\stackrel{g}\to \frac{2\mathbb{Z}}{16\mathbb{Z}}\to 0$, where $f(x+16\mathbb{Z}) = x+16\mathbb{Z}$ and $g(x+16\mathbb{Z})=4x+16\mathbb{Z}$ is an $e$-split, because we have a map $f_1:\frac{2\mathbb{Z}}{16\mathbb{Z}}\to{\frac{\mathbb{Z}}{16\mathbb{Z}}}, f_1(x+16\mathbb{Z})=x+16\mathbb{Z}$ such that $g\circ f_1(x+16\mathbb{Z})=g(x+16\mathbb{Z})=4x+16\mathbb{Z}=4I_\frac{2\mathbb{Z}}{16\mathbb{Z}}$. Thus we obtain $\frac{2\mathbb{Z}}{16\mathbb{Z}}$ is an $e$-projective as $\frac{\mathbb{Z}}{16\mathbb{Z}}$-module, while it is not projective.
\end{example}
\begin{definition}
	An e-projective resolution of an $ R- $module $ A $ is an e-exact sequence $ \dots\to P_{n+1}\stackrel{d_{n+1}} \to P_n \dots \stackrel{d_2}\to P_1\stackrel{d_1} \to P_0\stackrel{d_0} \to A \to 0 $ where each $P_n$ is an e-projective $R$-module. If $T$ is a contravariant functor, then $(R^nT)A\:=H^n(TP_A)=\frac{KerTd_{n+1}}{ImTd_n}$, where $ \dots \to P_{n+1} \to P_n\to \dots \to P_1 \stackrel{d_1}\to P_0 \to 0:P_A $ is a deleted $e$-projective resolution of an $R$-module $A$. In particular, we put $T= Hom(\: ,B)$, and define $_eext^n_R(\: ,B)=R^nT$. Then $_eext^n_R(A,B)=H^n(Hom_R(P_A,B))$, which means $_eext^n_R(A,B)=\frac{Kerd^{n^*}}{Im d^{(n-1)^*}}$,\\ 
where $d^{n^*}:Hom(P_{n-1},B)\to Hom(P_n,B)$ is defined as usual by $d^{n^*}:f\longmapsto fd^{n}$ 
\end{definition}
\begin{theorem}\label{ne}
	Let $A$ be an $R$-module. Then $_eext^n_R(A,B)=0$ for all negative integer $n$.
\end{theorem}
\begin{proof}
	Suppose that $ \dots \to P_{n+1} \to P_n\to \dots \to P_1 \to P_0 \to A \to 0:P $ be an $e$-projective resolution for $A$. Then the deleted complex of $A$ is $ \dots \to P_{n+1} \to P_n\to \dots \to P_1 \to P_0\to 0:P_A $, after applying $Hom(\:,B)$ on the deleted complex, we get $0\to Hom(P_0,B)\to Hom(P_1,B)\to Hom(P_2,B)\to\dots $by \cite[Theorem 2.7]{AZ}, which implies that $Hom(P_n,B)=0$ for all negative integer number $n$. Hence $_eext^n_R(A,B)=0$ for all negative integer number $n$. 
\end{proof}
\begin{theorem}\label{co}
	Let $A$ be $R$-module. Then $_eext^0_R(A,B)\cong rHom(A,B)$ for some $0\ne r\in R$.
\end{theorem}
\begin{proof}
	Let $ \dots \to P_{n+1} \to P_n\to \dots \to P_1\stackrel{d_1} \to P_0 \stackrel{\epsilon}\to A \to 0 $ be an $e$-projective resolution for $A$. By definition $_eext^0_R(A,B)= H^0(Hom(P_A,B))=\frac{Kerd_1^*}{Imd_\epsilon^*}=Kerd_1^*$. But left $e$-exactness of $Hom(\:,B)$ gives an $e$-exact sequence $0\to Hom(A,B)\stackrel{\epsilon^*}\to Hom(P_0,B)\stackrel{d_1^*} \to Hom(P_1,B)\stackrel{d_2^*}\to Hom(P_2,B)\to\dots $ by \cite[Proposition 2.7]{MA}. We define $\epsilon^*:Hom(A,B)\to Kerd_1^*$, since $Im\epsilon^* \leqslant_e Kerd_1^*$, $\epsilon^*$ is well-defined and since $Hom(\:,B)$ is a left $e$-exact functor then $\epsilon^*$ is monic. Now, we want to prove that $\epsilon^*$ is an epic. Let $f\in Ker d_1^*$ where $f:P_0\to B$, then $d^*_1(f(p_0))=f(d_1(p_1))=0$. We have $Im\epsilon\leqslant_e A$ so there exist $a'\in A$ and $0\ne r\in R$ such that $\epsilon (p_0)=ra'$. Now, we define $g:A\to B$ by $rg(a)=b$ for fixed $r\in R$. Let $a_1, a_2\in A$ and $a_1=a_2$. Then $ra_1=ra_2$ implies that $rg\epsilon(p_1)=rg\epsilon(p_2)$ so we botain $f(p_1)=f(p_2)$ and $b_1=b_2$ thus $g$ is well-defined. Now, $rf(p_0)=rg\epsilon(p_0)=g\epsilon(rp_0)=\epsilon^*(g(a))$. Hence $\epsilon^*$ is an isomorphism to $rKerd_1^{*}$ because, ${_eext^0_R(A,B)}=Ker d_1^*$ and ${_eext^0_R(A,B)}$ is isomorphic to $rHom(A,B)$.
\end{proof}
\begin{theorem}
	Let $P$ be an $e$-projective $R$-module, then $_eext^n_R(P,B)=0$, for all $n\geq1$. 
\end{theorem}
\begin{proof}
	Since $P$ is an $e$-projective, the $e$-projective resolution is $0\to P\stackrel{1_P}\to P\to 0$. The corresponding deleted $e$-projective resolution $P_P$ is $0\to P\to 0$. By applying $Hom(\:,B)$ to the deleted complex we obtain $_eext^n_R(P,B)=0$ for all $n\geq 1$.
\end{proof}
\begin{corollary}\label{long}
	Let $0\to A'\to A\to A''\to 0$ be an $e$-exact sequence of $R$-modules , then there is a long $e$-exact sequence $0\to r_1Hom(A'',B)\to r_2Hom(A,B)\to r_3Hom(A',B)\to {_eext^1_R(A''.B)}\to {_eext^1_R(A,B)}\to \dots $ for some nonzero element $r_1, r_2, r_3\in R$.
\end{corollary}
\begin{proof}
By \cite[Theorem 3.7]{ZA}, we have an $e$-exact sequence of deleted complexes $0\to P_{A'}\to P_A \to P_{A''}\to 0$ . If $T=Hom(\:,B)$, then $0\to TP_{A''}\to TP_A \to TP_{A'}\to 0$ is still $e$-exact by \cite[Proposition 2.7]{MA}. Then by \cite[Theorem 3.2]{ZA} we have an $e$-exact sequence  $0\to H^0(Hom(P_{A''},B))\to H^0(Hom(P_A,B))\to H^0(Hom(P_{A'},B))\to H^1(Hom(P_{A''}.B))\to H^1(Hom(P_A,B))\to \dots $. By using the definition of $_eext_n^R$, Theorem \ref{co} and Theorem \ref{ne} we obtain an $e$-exact sequence   $0\to r_1Hom(A'',B)\to r_2Hom(A,B)\to r_3Hom(A',B)\to {_eext^1_R(A''.B)}\to {_eext^1_R(A,B)}\to \dots $ for some nonzero element $r_1, r_2, r_3\in R$.
\end{proof}
\begin{theorem}
	Given a commutative diagram of $R$-modules having $e$-exact rows as the following:
		\[ \begin{tikzcd}[arrows={-Stealth}]
		0\rar & A'\rar["i"]\dar["f"] & A\rar["p"]\dar["g"] & A''\rar \dar["h"] & 0\\
		0\rar & C'\rar["j"]& C\rar["q"] & C''\rar & 0
	\end{tikzcd}	\] Then there is a commutative diagram of $R$-modules with $e$-exact rows:
	\[ \begin{tikzcd}[arrows={-Stealth}]
	{_eext^n_R(A'',B)}\rar["p*"] & {_eext^n_R(A,B)}\rar["i^*"] & {_eext^n_R(A',B)}\rar["\sigma^n"] & {_eext^{n+1}_R(A'',B)}\\
	{_eext^n_R(C'',B)}\rar["q*"]\uar["h^*"] & {_eext^n_R(C,B)}\rar["j^*"]\uar["g^*"]& {_eext^n_R(C',B)}\rar["\sigma'^n"]\uar["f^*"] & {_eext^{n+1}_R(C'',B)}\uar["h^*"]
\end{tikzcd}	\],
\end{theorem}
\begin{proof}
By \cite[Theorem 3.7]{ZA}, we have an $e$-exact sequence of deleted complexes $0\to P_{A'}\to P_A \to P_{A''}\to 0$ . If $T=Hom(\:,B)$, then $0\to TP_{A''}\to TP_A \to TP_{A'}\to 0$ is still $e$-exact by \cite[Proposition 2.7]{MA}. By \cite[Remark 3.3]{ZA} there is a commutative diagram of $R$-modules and $R$-morphisms as the following
	\[ \begin{tikzcd}[sep=0.2cm]
		\dots \to H^{n-1}(Hom(P_{A''},B)\rar["i^*"] & H^{n-1}(Hom(P_{A},B))\rar ["\sigma"] & H^n(Hom(P_{A'},B)\rar & \dots \\%
		\dots \to H^{n-1}(Hom(P_{C''},B)\rar["j^*"]\uar["g^*"] &	H^{n-1}(Hom(P_{C},B))\rar["\sigma^*"]\uar["h^*"] & H^n(Hom(P_{C'},B))\rar\uar["f^*"]&\dots 
	\end{tikzcd}	\] and our proof will be complete by using the definition of ${_eext^n_R(A,B)}=H^n(Hom(P_A,B))$.
\end{proof}
\subsection{Covariant essential derived functor $_eExt$}
In this subsection we want to describe covariant right derived functors $Ext^n_R$ on the $e$-injective resolution call it covariant essential derived functors (breifly $_eExt_n^R$) and discuss some properties of them and prove some results under acceptable conditions. In homology, $ext^R_n$ and $Ext^R_n$ are equivalent but this is not the case for $_eext^R_n$ and $_eExt^R_n$. We begin with the following definition.
\begin{definition}
	An e-injective resolution of an $ R- $module $ A $ is an e-exact sequence $ 0\to A\stackrel \eta \to E^0 \stackrel{d^{0}}\to E^1 \stackrel{d^{1}}\to ...\to E^n \stackrel{d^{n}}\to E^{n+1} \to ... $ where each $E^i$ is an e-injective $R$-module. If $T$ is a covariant functor,then $(R^nT)A=H^n(TE^M)=\frac{KerTd^n}{ImTd^{n-1}}$, where $E^A:0 \to E^0\stackrel {d^0} \to E^1 \stackrel {d^1}\to \dots$ is deleted e-injective resolution of an $R$-module $M$. In particular, we put $T= Hom(A,\: )$, for any $R$-module $A$, we define $_eExt^n_R(A,\:)=R^nT$. Then $_eExt^n_R(A,M)=H^n(Hom_R(A,E^M))$, which means $_eExt^n_R(A, M)=\frac{Kerd^{n}_*}{Im d_*^{n-1}}$,\\
	 where $d^{n}_*:Hom(A,E^n)\to Hom(A,E^{n+1})$ is defined as usual by $d^{n}_*:f\longmapsto d^{n}f$ 
\end{definition}
\begin{theorem}
	Let $E$ be an $e$-injective $R$-module, then for $R$-module $A$, $_eExt^n_R(A,E)=0$, for all $n\geq1$ 
\end{theorem}
\begin{proof}
	Since $E$ is an $e$-injective module, the $e$-injective resolution of $E$ is $0\to E\stackrel{1_E}\to E\to 0$. The corresponding deleted $e$-injective resolution $E^E$ is $0\to E\to 0$. By applying $Hom(\:,E)$ to the deleted complex we obtain $_eExt^n_R(A,E)=0$ for all $n\geq 1$.
\end{proof}
\begin{corollary}
	Let $0\to A''\to A\to A'\to 0$ be a short $e$-exact sequence of $R$-modules and $P$ be an $e$-projective module, then there is a long $e$-exact sequence $0\to Hom(P,A'')\to Hom(P,A)\to Hom(P,A')\to {_eExt^1_R(P,A'')}\to {_eExt^1_R(P,A)}\to \dots $
\end{corollary}
\begin{proof}
	By \cite[Proposition 2.10]{MA}, we have an $e$-exact sequence of deleted complexes $0\to E''^{A''}\to E^A \to E'^{A'}\to 0$ . If $T=Hom(P,\:)$, then $0\to TE^{A''}\to TE^A \to TE^{A'}\to 0$ is still $e$-exact by \cite[Theorem 3.1]{AZ}. Then by \cite[Theorem 3.2]{ZA} we have an $e$-exact sequence\\  $0\to H^0(Hom(P,E^{A''}))\to H^0(Hom(P,E^A))\to H^0(Hom(P,E^{A'}))\to H^1(Hom(P,E^{A''}))\to H^1(Hom(P,E^A))\to \dots $. By using the definition of $_e$Ext and \cite[Theorem 5.3]{ZA} we obtain an $e$-exact sequence  $0\to Hom(P,A'')\to Hom(P,A)\to Hom(P,A')\to {_eExt^1_R(P,A'')}\to {_eExt^1_R(P,A)}\to \dots $
\end{proof}
\begin{theorem}
	Let $P$ be an $e$-projective module and given a commutative diagram of $R$-modules having $e$-exact rows as the following:
	\[ \begin{tikzcd}[arrows={-Stealth}]
		0\rar & A''\rar["i"]\dar["f"] & A\rar["p"]\dar["g"] & A'\rar \dar["h"] & 0\\
		0\rar & C''\rar["j"]& C\rar["q"] & C'\rar & 0
	\end{tikzcd}	\] Then there is a commutative diagram of $R$-modules with $e$-exact rows:
	\[ \begin{tikzcd}[arrows={-Stealth}]
		{_eExt^n_R(P,A'')}\rar["i*"]\dar["f^*"] & {_eExt^n_R(P,A)}\rar["p^*"]\dar["g^*"] & {_eExt^n_R(P,A')}\rar["\sigma^n"]\dar["h^*"] & {_eExt^{n+1}_R(P,A'')}\dar["f^*"]\\
		{_eExt^n_R(P,C'')}\rar["j*"] & {_eExt^n_R(P,C)}\rar["q^*"]& {_eExt^n_R(P,C')}\rar["\sigma'^n"] & {_eExt^{n+1}_R(P,C'')}
	\end{tikzcd}	\]
\end{theorem}
\begin{proof}
	By \cite[Proposition 2.10]{MA}, we have an $e$-exact sequence of deleted complexes $0\to E^{A''}\to E^A \to E^{A'}\to 0$ . If $T=Hom(P,\:)$, then $0\to TE^{A''}\to TE^A\to TE^{A'}\to 0$ is still $e$-exact by \cite[Theorem 3.1]{AZ}. By \cite[Remark 3.3]{ZA} there is a commutative diagram of $R$-modules and $R$-morphisms as the following
	\[ \begin{tikzcd}[sep=0.2cm]
		\dots \to H^{n-1}(Hom(P,E^A))\rar["p^*"]\dar["g^*"] & H^{n-1}(Hom(P,E^{A'}))\rar ["\sigma"]\dar["h^*"] & H^n(Hom(P,E^{A''}))\rar \dar["f^*"]& \dots \\%
		\dots \to H^{n-1}(Hom(P,E^{C}))\rar["q^*"] &	H^{n-1}(Hom(P,E^{C'}))\rar ["\sigma'^*"]\rar & H^n(Hom(P,E^{C''}))\rar&\dots 
	\end{tikzcd}	\] and our proof will be complete by using the definition of ${_eExt^n_R(P,A)}=H^n(Hom(P,E^A))$.
\end{proof}

\section{ideal transforms regarding to essential exact sequences}
 Throughout this section $R$ be a hereditary domain. All our work in this section applies to two particular systems of ideals the system of ideals $\beta=(a^n)_{n\in N}$ and the system of ideals $\beta=(a^n+b^n)_{n\in N}$by \cite[Example 3.12]{MR}. Then the ideal transforms with respect to an ideal $a$ is defined as $D_a(B)=\varinjlim_{n\in N} Hom(a^n,B)$. It is a covariant $R$-linear functor as well as it is left $e$-exact sequence because $Hom(a^n,B)$ is a left $e$-exact sequence and direct limit preserves $e$-exactness. If a system of ideals $(a^n)_{n\in N}$ is an inverse family of ideals, then there is a natural equivalence $\varinjlim_{n\in N} Hom(\frac{R}{a^n},B)\cong \Gamma_a(B)$ as well as we have a natural equivalent between $\varinjlim_{n\in N} {_eext^i_R(\frac{R}{a^n},B)}$ and ${_eH^i_a(B)}$. Let $0\to A \to B \to C\to 0$ be an $e$-exact sequence. If $C$ is an $e$-projective, then by \cite[Proposition 3.2]{AZ} the $e$-exact sequence $0\to A \to B \to C\to 0$ is an $e$-split and then by \cite[Proposition 2.5]{MA} $rB\cong A\oplus C$. Therefore, there is a submodule $C$ of $B$ with $C\cong \frac{rB}{A}$.
\begin{theorem}\label{RD}
	 For any torsion-free $R$-module $B$. The sequence $0\to \Gamma_{a}(B) \to B\to D_{a}(B)\to H^1_{a}(B)\to 0$ is an $e$-exact.
\end{theorem}
\begin{proof}
	The sequence $0\to a^n \to R\to \frac{rR}{a^n}\to 0$, where $0\ne r\in R$ is an $e$-exact sequence. From above $e$-exact sequence and Corollay \ref{long} induces a long $e$-exat sequence of $_eext^n_R(\:,B)$ modules. Since $R$ is an $e$-projective and since $Hom(R,B)$ is naturally isomorphic to $B$, we therefore obtain a long $e$-exact sequence  $0\to Hom (\frac{rR}{a^n},B)\to B\to Hom(a^n,B)\to {_eext^1_R(\frac{rR}{a^n}},B)\to 0$. Now pass to direct limits $0\to \varinjlim_{n\in N}Hom (\frac{rR}{a^n},B)\to B\to \varinjlim_{n\in N}Hom(a^n,B)\to \varinjlim_{n\in N}{_eext^1_R}(\frac{rR}{a^n},B)\to 0$, then use the natural equivalent to obtain an $e$-exact sequences of $R$-modules and $R$-morphism $0\to \Gamma_{a}(B) \to B\to D_{a}(B)\to H^1_{a}(B)\to 0$. 
\end{proof}
\begin{theorem}\label{hom}
	Let the ideal $a^n$ be a torsion-free module for all natural number $n$. Then $Hom(a^n,B)$ is an $e$-injective.
\end{theorem}
\begin{proof}
	To prove that $Hom(a^n,B)$ is an $e$-injective, we show that\\ $Hom(-, Hom(a^n,B))$ is an $e$-exact functor. By the adjoint isomorphism, this functor is naturally isomorphic to $Hom(a^n\otimes-,B)$ which is the composite $Hom(-,B)\circ(a^n\otimes-)$. By \cite[Proposition 2.7]{MA}, $Hom(-,B)$ is an $e$-exact functor and by \cite[Theorem 2.10]{AZ}, $a^n\otimes-$ is also $e$-exact, so their composite is again $e$-exact.
\end{proof}
\begin{theorem}
For any torsion-free $R$-module $B$ there exists $0\ne r\in R$ such that $\epsilon^*: B\to rD_{a}(B)$ is an isomorphism if and only if $\Gamma_{a}(B)= H^1_{a}(B)=0$.
\end{theorem}
\begin{proof}
	 If $\epsilon^*$ is an isomorphism then by $e$-exactness and definition of essential we get the result. Conversely, suppose that $\Gamma_{a}(B)= H^1_{a}(B)=0$ then by $e$-exactness and definition of essential we obtain $Ker \epsilon^*=0$ implies $\epsilon^*$ is monic. It is remain to show that $\epsilon^*$ is an epic. By Theorem \ref{hom} $D_{a}(B)$ is an $e$-injective and by \cite[Proposition 2.8]{MA} there exists a homomorphism $g:Im\epsilon^*\to D_{a}(B)$ such that $f_3\circ g=r I_{D_{a}(B)}$ which means that $rD_{a}(B)=Im\epsilon^*$. Therefore, $\epsilon^*$ is an isomorpism.
	 \[\begin{tikzcd}[sep=2cm]
	 	{} & D_{a}(B) & {} \\
	 	0 \ar{r} & B \ar{u}{\epsilon^*} \ar[r,swap]{r}{f_1} & Im\epsilon^* \ar[dashed,swap]{lu}{f_3}\ar[lu,<-,shift left=2,"g"]
	 \end{tikzcd}\]
\end{proof}
\begin{theorem}\label{3.4}
	Let ${_eext^n_R(a^n,B)}$ be an $e$-injective torsion-free module. Then there exists $0\ne r\in R$ such that $\eta: {_eext^n_R(a^n,B)}\to r{_eext^{n+1}_R(\frac{rR}{a^n},B)}$ is an isomorphism.
\end{theorem}
\begin{proof}
	  Corollay \ref{long} induces a long $e$-exact sequences of $_eext^n_R$ and by $e$-exactness and definition of essential we obtain $Ker \eta=0$ implies $\eta$ is monic. It is remain to show that $\eta$ is an epic. By \cite[Proposition 2.8]{MA} there exists a homomorphism $g:Im\eta\to {_eext^{n+1}_R(\frac{rR}{a^n},B)}$ such that $f_3\circ g=r I_{{_eext^{n+1}_R(\frac{rR}{a^n},B)}}$ which means that $r{_eext^{n+1}_R(\frac{rR}{a^n},B)}=Im\eta$. Therefore, $\eta$ is an isomorpism.
	   \[\begin{tikzcd}[sep=2cm]
	  	{} & {_eext^{n+1}_R(\frac{rR}{a^n},B)} & {} \\
	  	0 \ar{r} & {_eext^n_R(a^n,B)} \ar{u}{\eta} \ar[r,swap]{r}{f_1} & Im\eta \ar[dashed,swap]{lu}{f_3}\ar[lu,<-,shift left=2,"g"]
	  \end{tikzcd}\]
\end{proof}
 The Mayer-Vietoris sequence involves two ideals so the $a$ be the first ideal and b will denote a second ideal. In the following theorem we have the new $e$-exact sequences which is obtain by generalize the idea of Mayer-vietores sequence.
\begin{theorem}\label{SHD}
	For any torsion-free $R$-module $B$, there is an $e$-exact sequence $0\to D_{r(a+b)}(B)\to D_a(B)\bigoplus D_b(B)\to D_{a\cap b}(B)\to rH^2_{r(a+b)}(B)\to rH^2_{a}(B)\bigoplus rH^2_{b}(B)\to rH^2_{a\cap b}(B)\to \dots$ for some $0\ne r\in R$.
\end{theorem}
\begin{proof}
	We have an $e$-exact sequences $0\to a\cap b\to a\bigoplus b\to \frac{r(a\bigoplus b)}{a\cap b}\to 0$, where $0\ne r\in R$. By Corollay \ref{long} induces a long $e$-exat sequences $0\to Hom (r(a+b),B)\to Hom(a,B)\bigoplus Hom(b,B)\to Hom(a\cap b,B)\to {_eext^1_R(r(a+b),B)}\to {_eext^1_R(a,B)}\bigoplus {_eext^1_R(b,B)}\to {_eext^1_R(a\cap b,B)}\to \dots$
	 Now pass to direct limits \\
	 $0\to\varinjlim_{n\in N}Hom(r(a+b),B)\to \varinjlim_{n\in N} Hom(a,B)\bigoplus\varinjlim_{n\in N}Hom(b,B)\to\varinjlim_{n\in N}Hom(a\cap b,B)\to\varinjlim_{n\in N}{_eext^1(r(a+b),B)}\to
	 \varinjlim_{n\in N}{_eext^1_R(a,B)}\bigoplus\\\varinjlim_{n\in N}{_eext^1_R(b,B)}\to\varinjlim_{n\in N}{_eext^1_R(a\cap b,B)}\to\dots$, then by using a natural equivalent and Theorem \ref{3.4} induces an $e$-exact sequences\\ $0\to D_{r(a+b)}(B)\to D_a(B)\bigoplus D_b(B)\to D_{a\cap b}(B)\to rH^2_{r(a+b)}(B)\to rH^2_{a}(B)\bigoplus rH^2_{b}(B)\to rH^2_{a\cap b}(B)\to \dots$.
\end{proof}
\vskip 0.4 true cm

\end{document}